\documentclass[12pt]{amsart}
\usepackage[top=1in, bottom=1in, left=1in, right=1in]{geometry}

\usepackage{amssymb,mathrsfs,amsmath}

\newtheorem{theorem}{Theorem}[section]
\newtheorem{proposition}{Proposition}[section]
\newtheorem{lemma}[theorem]{Lemma}
\newtheorem{corollary}[theorem]{Corollary}
\theoremstyle{remark}

\newcommand{\bsp}{\begin{split}}
\newcommand{\esp}{\end{split}}
\newcommand{\be}{\begin{equation}}
\newcommand{\ee}{\end{equation}}
\newcommand{\bes}{\begin{equation*}}
\newcommand{\ees}{\end{equation*}}
\newcommand{\bv}\boldsymbol{}

\numberwithin{equation}{section}

\renewcommand{\pmod}[1]{~({\rm mod}\,#1)}

\begin{document}

\title{Large character sums: Burgess's theorem and zeros of $L$-functions} 
\author{Andrew Granville} 
\address{AG: D\'epartement de math\'ematiques et de statistique,
Universit\'e de Montr\'eal, CP 6128 succ. Centre-Ville,
Montr\'eal, QC H3C 3J7, Canada.; and Department of Mathematics, University College London, Gower Street, London WC1E 6BT, England.}
\email{{\tt andrew@dms.umontreal.ca}}
\author{Kannan Soundararajan} 
\address{KS: Department of Mathematics, Stanford University, Stanford, CA 94305, USA.} 
\email{\tt ksound@stanford.edu} 
\thanks{Andrew Granville was partially supported by an NSERC Discovery Grant, a Canadian Research Chair, and an ERC Advanced Grant. Kannan Soundararajan was partially supported by NSF grant DMS 1001068, and a Simons Investigator grant from the Simons Foundation.}
\date{\today}

\begin{abstract} We study the conjecture that $\sum_{n\leq x} \chi(n)=o(x)$ for any primitive Dirichlet character $\chi \pmod q$ with $x\geq q^\epsilon$, which is known to be true if the Riemann Hypothesis holds for $L(s,\chi)$. We show that it holds under the weaker assumption that `$100\%$' of the zeros of $L(s,\chi)$ up to height $\tfrac 14$ lie on the 
critical line. We also establish various other consequences of having large character sums; for example, that if  the conjecture holds for $\chi^2$ then it also holds for $\chi$.
\end{abstract}

\maketitle 

\section{Introduction} 

\noindent A central quest of analytic number theory is to estimate the 
character sum
\begin{equation} 
\label{1.1} 
S(x,\chi) = \sum_{n\le x} \chi(n),
\end{equation} 
where $\chi\pmod q$ is a primitive character.    We would like to show that 
\begin{equation} 
\label{Conj}  S(x,\chi)=o(x)
\end{equation}   
in as wide a range for $x$ as possible, and in particular 
whenever  $x\geq q^\epsilon$ for any fixed positive $\epsilon$ (which implies Vinogradov's conjecture that the least quadratic non-residue mod $q$ is $\ll_\epsilon q^\epsilon$). In \cite{GSLCS1} we
showed that  \eqref{Conj} holds when $\log x/\log\log q \to \infty$, assuming the Riemann Hypothesis for $L(s,\chi)$,
and proved unconditionally that this range is the best possible.  
Burgess \cite{Bu1,Bu2} gave the best unconditional result, now more than fifty years old, that \eqref{Conj} holds for all $x\geq q^{1/4+\epsilon}$ 
if $q$ is assumed to be cube-free, and slightly weaker variants for general $q$.   The main results of this paper give further connections between large values of character sums and zeros of the corresponding $L$-function. 

Before describing our two main theorems, we give the following corollaries which give a qualitative feel 
for what is established.


\begin{corollary}\label{cor4}  Let $\chi$ be a primitive quadratic character $\pmod q$, let $\epsilon> (\log q)^{-\frac 13}$ be 
real.  If the region $\{ s: \ \text{Re}(s) \ge \frac 34, \ |\text{Im}(s)| \le \frac 14\}$ contains 
no more than $\epsilon^2 (\log q)/1600$ zeros of $L(s,\chi)$, then for all $x\ge q^{\epsilon}$ we have 
\[
 \Big| \sum_{n\le x} \chi(n) \Big|  \ll \frac x{(\log x)^{\frac{1}{100}} }.
 \]
\end{corollary}

There are $\ll \log q$ zeros $\beta+i\gamma$ of $L(s,\chi)$ with $0<\beta<1$ and 
$|\gamma|\leq \frac 14$, and we expect these zeros to satisfy the Riemann hypothesis $\beta=\frac 12$. 
Our result can therefore be paraphrased as stating that if \eqref{Conj} is false for $x=q^\epsilon$ then a positive proportion ($\gg \epsilon^2$) of  the zeros of $L(s,\chi)$, up to height $1$, lie off the $\frac 12$-line.  We believe that the method could be adapted to 
increase this proportion to $\gg \epsilon^{1+\delta}$ for any $\delta >0$, but we do not pursue this here. 
A similar result holds for arbitrary primitive characters.  

\begin{corollary}\label{cor3} Let $\chi \pmod q$ be a primitive character.  Let $\epsilon$ and $T$ be real numbers with $1\le T\le (\log q)^{\frac 1{200}}$ and $\epsilon \ge (\log q)^{-\frac 13}$. 
Suppose that for every real $\phi$ with $|\phi|\leq T$ the region 
$\{ s: \ \text{Re}(s) \ge \frac 34, \ |\text{Im}(s)-\phi| \le \frac 14\}$ contains no more than $\epsilon^2 (\log q)/1440$ 
zeros of $L(s,\chi)$.  Then for all $x\ge q^{\epsilon}$ we have 
\[
 \Big| \sum_{n\le x} \chi(n) \Big| \ll \frac x{T} .
\]
\end{corollary}


We now state our main theorems, from which the corollaries above follow as special cases.

\begin{theorem}  \label{main1} 
Let $\chi\pmod q$ be a primitive character, and let $\exp(\sqrt{\log q}) \le x\le \sqrt{q}$ be such that 
$|S(x,\chi)| = x/N$ where $1\leq N\leq  (\log x)^{\frac 1{100}}$.  There exists an absolute positive constant $c>0$ such that for some real number $\phi$ 
with $|\phi|\le c N$, and any parameter $(\log x)/2 \ge L \ge cN^6$, the region 
$$ 
\Big \{ s: \ |s-(1+i\phi)| <  L\frac{\log q}{(\log x)^2} \Big\} 
$$ 
contains at least $L/360$ zeros of the Dirichlet $L$-function $L(s,\chi)$. 
\end{theorem}

When the character $\chi$ has small order, we give the following variant which removes the parameter $\phi$ in Theorem \ref{main1}. 

\begin{theorem} \label{main2}  Let $\chi \pmod q$ be a primitive Dirichlet character of order $k$, and let $x$, and $N$ be 
as in Theorem \ref{main1}.   There exists an absolute constant $c>0$ such that for any parameter $L$ in the range $(\log x)/2 \ge L \ge (cN)^{2k^2}$, the region 
$$ 
\Big \{ s: \ |s-1| <  L\frac{\log q}{(\log x)^2} \Big\} 
$$ 
contains at least $L/400$ zeros of $L(s,\chi)$.  
\end{theorem}  


Our first corollaries showed that large character sums produced many violations to the GRH.  Our next corollary shows that large character sums force some zeros of $L(s,\chi)$ to lie very close to the $1$-line (refining an old result of Rodosski\v i \cite{Rod} who treated the related problem of determining the smallest prime $p$ with $\chi(p)\neq 1$, for characters of small order; see \cite{Mo} for a lucid exposition).  


\begin{corollary} \label{cor2}   Let $1\ge \epsilon \ge (\log q)^{-\frac{1}{200}}$, and suppose there exists $x\ge q^{\epsilon}$ with $|S(x,\chi)| \ge \epsilon x$.  Then there is an absolute constant $c>0$ such that there is at least one zero of $L(s,\chi)$ inside the region 
\[
\Big \{ s : \ \text{Re}(s) \ge 1- \frac{c}{\epsilon^{8}\log q},  \  |\text{Im}(s)| \le \frac{c}{\epsilon}\Big \} .
\]
If $\chi$ has order $k$ and $\epsilon \ge (\log q)^{-\frac{1}{4k^2}}$ then there is a zero in the region 
$|s-1| \le c/(\epsilon^{2k^2+2}\log q)$.
\end{corollary}

A classical argument of Backlund (see Theorem 13.5 of \cite{Ti}) could be adapted to show that if 
almost all the zeros of $L(s,\chi)$ in intervals of length $1$ have real part $\le 1/2+ \epsilon$ then 
the Lindel{\" o}f hypothesis for $L(s,\chi)$ would follow.  From such a bound, one could obtain 
strong estimates for character sums.   Our results provide a sharper version of such ideas of 
Backlund and Rodosski\v i, by finding zeros even closer to the $1$-line, and localizing their imaginary 
parts.  A key ingredient in our argument is work on mean-values of multiplicative functions, in particular 
the feature that such mean values vary slowly (see Lemma 3.3 below). 

\medskip 

By a compactness argument, another consequence of our work is that if \eqref{Conj} fails for $x\ge q^{\epsilon}$ 
for infinitely many characters of bounded order, then one can find a sequence of $L$-functions with arbitrarily many pinpointed zeros near the $1$-line.   

\begin{theorem} \label{thm1.6}  Fix an integer $k \ge 2$ and a constant $\eta >0$.  Suppose there is an infinite sequence of distinct primitive characters $\chi_j \pmod{q_j}$ of order $k$ for which $| \sum_{n\le x} \chi_j(n)| \ge \eta x$ 
for some $x\ge q_j^{\eta}$.  There exists an infinite sequence $z_1$, $z_2$, $\ldots$, of complex numbers 
with $|z_n|+1 \le |z_{n+1}|$ with the following property:  There is a sequence of primitive characters $\psi_j \pmod {r_j}$ (in fact a subsequence of the original sequence $\chi_j$) with $L(s_{\ell}, \psi_j)=0$ for $1\le \ell \le j$ and 
some $s_{\ell}$  satisfying 
$$ 
s_{\ell}  = 1+ \frac{z_{\ell} +o(1)}{\log r_j},
$$ 
and the $o(1)$ term tends to zero as $r_j \to \infty$.  
\end{theorem}

 Theorem \ref{thm1.6} generalizes and gives a soft version of an unpublished observation of Heath-Brown.  
 Heath-Brown observed that if there is an infinite sequence of primes $q$ for which the least quadratic non-residue mod $q$ is $\geq q^{1/4\sqrt{e}+o(1)}$ then one can locate precisely  many zeros of $L(s,(\frac{\cdot}q))$.  
 A precise version of his result, as described in Appendix 2 of \cite{DMV}, is as follows.  Consider the zeros of 
\[
H(z)= \frac 2z \int_{1/\sqrt{e}}^1 (1-e^{-zu}) \frac{du}u.
\]
These zeros lie in the half plane Re$(z) <0$, and occur in conjugate pairs.  Let $z_k$ denote the sequence of these zeros with positive imaginary parts, and arranged in ascending order of the imaginary part.  
For each $k\geq 1$, if $q$ is sufficiently large (and the least quadratic non-residue is as large as $q^{\frac{1}{4\sqrt{e}}+o(1)}$),  then there is a zero of $L(s,(\frac{\cdot}{q}))$ at 
\[
 s=1+\frac{4z_k+o(1)}{\log q}, 
\]
and at its complex conjugate.  In this situation, one can also describe the zeros $z_k$ precisely:\    arguing as in Lemma 2 of \cite{DMV} gives that
\begin{equation} \label{ZeroLoc1}
 z_{k}= -\log (\pi k) + 2\pi i(k+\tfrac 14)  +o(1),
\end{equation}
which corresponds well to the data given at the end of \cite{DMV}.

Recently Banks  and Makarov \cite{BM} generalized Heath-Brown's observation, and 
showed that if there is a sequence of quadratic characters with a certain prescribed smooth way in 
which \eqref{Conj} fails, then one can pinpoint the zeros near $1$ of the corresponding $L$-functions.   The smoothness hypothesis that they assume, permits them to locate the zeros in a form similar to 
\eqref{ZeroLoc1} (see Proposition 3.1 of \cite{BM}).   In contrast, our Theorem \ref{thm1.6} is 
softer but holds more generally; it would be interesting if some more precise version of  Theorem \ref{thm1.6} 
incorporating behavior as in \eqref{ZeroLoc1} could be established.   We  note here the recent interesting work of Tao \cite{Tao} relating Vinogradov's conjecture to the Elliott-Halberstam conjectures on the distribution of primes (and more general sequences) in progressions.





We also take this opportunity to record some other observations on large character sums.   
In \cite{GSLCS2}, we proved that if $\chi_1,\chi_2,\chi_3$ are three (not necessarily distinct) characters $\pmod q$ which have large maximal character sums (that is, if  $\max_{x} |\sum_{n\leq x}
\chi_j(n) |\gg \sqrt{q} \log q$ for $j=1,2,3$, which is the largest size permitted by 
the Polya-Vinogradov theorem) then there exists some $x$ for which
\[ 
\Big|\sum_{n\leq
x_0} (\chi_1\chi_2\chi_3)(n) \Big|\gg \sqrt{q} \log q.
\]
We will prove an analogous (but much easier) result with respect to \eqref{Conj}.

\begin{corollary}\label{MultChars}  Suppose $\chi_1$, and $\chi_2$ are Dirichlet characters $\pmod q$, 
such that for some $x_1$, $x_2$ and some $\eta>0$ we have  (for $j=1$, $2$) 
$$
\Big|   \sum_{n\leq x_j} \chi_j(n)\Big| \ge \eta  x_j.
$$
Then, with $\xi = c\eta^6$ for a suitable absolute constant $c>0$, there exists $x\geq (\min (x_1,x_2))^\xi$ with 
$$
\Big|  \sum_{n\leq x} (\chi_1\chi_2)(n) \Big| \ge \xi x .
$$
\end{corollary}

 Corollary \ref{MultChars} implies, for example, that if $\chi \pmod{p}$ is a character of order $4$ for which
\eqref{Conj} fails for $x\ge p^{\epsilon}$, then \eqref{Conj} also fails for the Legendre symbol $\pmod p$ 
for some suitably large $x$.  We discuss Corollary \ref{MultChars} and related results in Section 6 below.

Given a prime $q$,  Burgess's theorem guarantees that there are $\sim x/2$ of quadratic residues and $\sim x/2$ quadratic non-residues $\pmod q$ up to $x$, provided $x\ge q^{\frac 14+o(1)}$.   One of the main results in \cite{GSSpectrum} shows that if $x$ is large enough, then at least $17.15\%$ of the integers below $x$ are 
quadratic residues $\pmod q$ (uniformly for all primes $q$).  In the `Vinogradov range' $q^{1/4\sqrt{e}+o(1)} \le x \le q$, Banks {\sl et al} (see Theorem 2.1 of \cite{BGHS})  showed that a positive proportion of the integers below $x$ are quadratic non-residues $\pmod q$.   We give  the following strengthening of their work (as mentioned in \S 4 of \cite{BGHS}).  

\begin{corollary} \label{quadres}  Let $q$ be a large prime, and suppose $1/\sqrt{e} \le u \le 1$.  The number of 
quadratic non-residues $\pmod q$ up to $x=q^{u/4}$ is at least 
$$ 
\Big( \min\Big(\delta_0, \frac 14- (\log u)^2 \Big) + o(1) \Big) x,
$$ 
where 
$$ 
\delta_0 = 1-\log (1+\sqrt{e}) + 2\int_{1}^{\sqrt{e}} \frac{\log t}{t+1} dt = 0.1715\ldots .
$$ 
\end{corollary}  



\section{Mean values of Multiplicative functions}\label{multfns}

\noindent In this section, we recall some results from the theory of mean-values of multiplicative functions.   
Let $f$ be a multiplicative function for which each $|f(n)|\leq 1$, and write $F(s) = \sum_{n=1}^{\infty} f(n)n^{-s}$. 
Define the (square of the) ``distance'' between two such functions $f$ and $g$,  
\[
\mathbb D(f,g;x)^2 := {\sum_{p\leq x} \frac{1-\text{Re }f(p)\overline{g(p)}}p  }, 
\]
and this distance function satisfies the triangle inequality
\begin{equation} \label{TrIneq}
 \mathbb D(f,g;x) + \mathbb D(g,h;x) \geq \mathbb D(f,h;x).
\end{equation}
Further, the distance function is related to the Dirichlet series $F(s)$ via the relation 
\begin{equation} 
\label{Ftodist}
 F\Big( 1 +\frac 1{\log x} +i t \Big) \asymp \log x \exp(-\mathbb D(f,n^{it};x)^2).
\end{equation} 

Given $x$, let $\phi=\phi_f(x)$ be a real number in the range $|t|\le \log x$ where $|F(1+1/\log x+it)|$ attains its 
maximum.  Put $M:=M_f(x)=\mathbb D(f,n^{i\phi};x)^2$.  The first fact that we need is Hal\'asz's Theorem (see, e.g., Theorem 2b of \cite{GSDecay}), which  gives
\begin{equation} \label{Hal1}
 \frac 1x \sum_{n\leq x} f(n) \ll \frac{(M+1)e^{-M}}{1+|\phi|} + \frac{1}{(\log x)^{2-\sqrt{3}+o(1)}} .
\end{equation}
Define $f_\phi(n):=f(n)/n^{i\phi}$.  We next need a relation between the mean value of $f$ and the mean value of $f_\phi$.  
From Lemma 7.1 of \cite{GSDecay}, we quote the relation  
\begin{equation} \label{Hal2}
 \frac 1x \sum_{n\leq x} f(n) = \frac{x^{i\phi}}{1+i\phi}\cdot  \frac 1x \sum_{n\leq x} f_\phi(n) +O\Big(
   \frac{\log \log x}{\log x} \exp\Big( \sum_{p\le x} \frac{|1-f(p)|}{p}\Big)\Big).
\end{equation}
Finally from Theorem 4 of \cite{GSDecay} (refining work of Elliott \cite{Elliott}) we require the following Lipschitz estimate 
showing that the mean values of $f_{\phi}$ vary slowly:   for any $\sqrt{x} \le z\le x^2$,  we have 
\begin{equation} \label{Hal3}
 \frac 1x \sum_{n\leq x} f_\phi(n) -  \frac 1{z} \sum_{n\leq z} f_\phi(n) \ll  \left(\frac{1+|\log x/z|}{\log x} \right)^{1-\frac{2}{\pi}+o(1)} .
\end{equation}

\section{Large character sums and  zeros off the critical line} 

\noindent Let $\chi \pmod q$ denote a primitive character.  We shall
 make use of the Hadamard factorization formula (see \cite{Dav}) 
 \begin{equation} 
 \label{Hadamard} 
 \xi(s,\chi) = \Big(\frac{q}{\pi} \Big)^{\frac{s+{\mathfrak a}}{2} } \Gamma\Big(\frac{s+{\mathfrak a}}{2} \Big) 
 L(s, \chi) = e^{A+Bs} \prod_{\rho} \Big(1-\frac{s}{\rho}\Big) e^{s/\rho}, 
 \end{equation} 
 where ${\mathfrak a}=(1-\chi(-1))/2$, and $\rho$ ranges over the non-trivial zeros of $L(s,\chi)$ 
 and $A$ and $B$ are constants (depending on $\chi$) with 
 \begin{equation} 
 \label{Hadamard2} 
 \text{Re } B = - \sum_{\rho} \text{Re }\frac{1}{\rho}.
 \end{equation}   

\begin{lemma}\label{Hadlem}  Let $\tfrac 12 \ge \lambda >0$ be a real number, and 
let $t$ be a real number.  Then 
$$ 
|L(1-\lambda+ it,\chi) |\ll \frac{1}{\lambda} \exp\Big( \sum_{\rho} \frac{2\lambda^2}{|1+\lambda+it -\rho|^2}\Big). 
$$ 
\end{lemma} 
  
\begin{proof}  Put $s_0= 1+\lambda+it$, and $s_1 = 1- \lambda+it$.  Applying \eqref{Hadamard} and \eqref{Hadamard2} 
 with $s=s_0$ and $s=s_1$, and invoking Stirling's formula, we see that 
  \begin{equation} 
  \label{2.3}  
  \Big| \frac{L(s_1,\chi)}{L(s_0,\chi)}\Big| \asymp (q(1+|t|))^{\lambda} \Big| \frac{\xi(s_1,\chi)}{\xi(s_0,\chi)}\Big| = (q(1+|t|))^{\lambda} 
  \prod_{\rho} \frac{|s_1-\rho|}{|s_0-\rho|}. 
  \end{equation} 
  Note that 
  \begin{align*}
   \Big| \frac{s_1-\rho}{s_0-\rho}\Big| &= \Big(1 - \frac{|s_0-\rho|^2-|s_1-\rho|^2}{|s_0-\rho|^2}\Big)^{\frac 12} 
  \le \exp\Big( -2\lambda \frac{\text{Re}(1-\rho)}{|s_0-\rho|^2} \Big) \nonumber\\
  &= \exp\Big( -2\lambda \text{Re }\Big(\frac{1}{s_0-\rho}\Big) 
 +\frac{2\lambda^2}{|s_0-\rho|^2}\Big) . 
  \end{align*} 
  Using this in \eqref{2.3}, we conclude that 
  \begin{equation} 
  \label{2.4} 
  \Big|\frac{L(s_1,\chi)}{L(s_0,\chi} \Big| 
  \ll (q(1+|t|))^{\lambda} \exp\Big(  2\lambda \sum_{\rho} \Big(-\text{Re }\Big(\frac{1}{s_0-\rho}\Big) + \frac{\lambda}{|s_0-\rho|^2}\Big)\Big). 
  \end{equation} 
  
On the other hand, taking logarithmic derivatives in \eqref{Hadamard}, we see that 
  \begin{equation} 
  \label{2.41}
 - \text{Re}  \frac{L^{\prime}}{L}(s_0,\chi) = \frac{1}{2} \log q(1+|t|) + O(1)  - \sum_{\rho} \text{Re }\Big(\frac{1}{s_0-\rho}\Big), 
\end{equation}
and the left hand side above is trivially bounded in magnitude by 
\begin{equation} 
\label{2.42} 
\le \sum_{n=1}^{\infty} \frac{\Lambda(n)}{n^{1+\lambda}} = \frac{1}{\lambda} +O(1).
\end{equation}
  Inserting this bound in \eqref{2.4}, we conclude that 
  $$ 
  \Big| \frac{L(s_1,\chi)}{L(s_0,\chi)}\Big| \ll \exp\Big( \sum_{\rho} \frac{2\lambda^2}{|s_0-\rho|^2}\Big),
  $$ 
  and since $|L(s_0,\chi)| \le \zeta(1+\lambda) = \frac{1}{\lambda}+O(1)$, the lemma follows.  
  \end{proof} 

Next, we show how character sums may be related to suitable averages of $L$-functions.  

\begin{lemma} \label{lem2.2}  Let  $\phi$ be a real number, $T$ a positive real number, and $0\le \lambda\le \frac 12$.  Let $\chi_{\phi}$ denote the completely multiplicative function $\chi_{\phi}(n) = \chi(n) n^{-i\phi}$, and let $S(x,\chi_{\phi}) = \sum_{n\le x} \chi_{\phi}(n)$.  Then 
$$ 
\sqrt{2\pi T} \int_{-\infty}^{\infty} \frac{S(e^{y},\chi_{\phi})}{e^y} \exp\Big({\lambda y - \frac{T}{2}y^2}  \Big) dy = \int_{-\infty}^{\infty} 
\frac{L(1-\lambda+i\phi+i\xi,\chi)}{1-\lambda+i\xi} \exp\Big(-\frac{\xi^2}{2T}\Big) d\xi.
$$ 
\end{lemma} 
\begin{proof}  The Fourier transform of $\frac{S(e^y,\chi_{\phi})}{e^y} \exp(y\lambda)$ is 
\begin{align*} 
\int_{-\infty}^{\infty}\frac{ S(e^y,\chi_{\phi})}{e^y}  \exp(y\lambda) e^{-iy\xi} dy
& =  \sum_{n=1}^{\infty} \frac{\chi(n)}{n^{i\phi}} \int_{\log n}^{\infty} \exp(-y(1-\lambda+i\xi)) dy\\ 
&=   \frac{L(1-\lambda+i\phi+i\xi,\chi)}{1-\lambda+i\xi}.
\end{align*}
The Fourier transform of $\exp(-\frac T2 y^2)$ is 
$$ 
\int_{-\infty}^{\infty} \exp\Big(-\frac T2 y^2-iy\xi\Big) dy  = \frac{\sqrt{2\pi}}{\sqrt{T}} \exp\Big(-\frac{\xi^2}{2T}\Big).  
$$  
The lemma follows by the Plancherel formula. 
\end{proof}  

Our last ingredient comes from the theory of mean-values of multiplicative functions.   

\begin{lemma}\label{lem2.3}  Let $y_0\ge 3$ be a real number, and we assume that $|S(e^{y_0},\chi)| \ge e^{y_0} y_0^{-\frac{1}{100}}$.  There exists a real number $\phi=\phi(y_0)$, with $|\phi| \ll e^{y_0}/|S(e^{y_0},\chi)|$, 
such that for any real number $y$,
\begin{equation} 
\label{lem231}
\Big| \frac{S(e^{y},\chi_{\phi})}{e^{y}}  - \frac{S(e^{y_0},\chi_{\phi})}{e^{y_0} } \Big| \ll \Big(\frac{1+|y-y_0|}{y_0}\Big)^{\frac 13} . 
\end{equation} 
Moreover 
\begin{equation} 
\label{lem232}
S(e^{y_0},\chi_{\phi}) = (1+i\phi)e^{-i\phi y_0} S(e^{y_0},\chi) + O\Big(\frac{e^{y_0}}{y_0^{\frac 34}}\Big). 
\end{equation}
Finally, if $\chi$ has small order $k$ then the following stronger bound for $\phi$ holds:  for some absolute constant $c>0$,    
\begin{equation} 
\label{lem233} 
|\phi| \le \frac{1}{y_0} \Big( \frac{ce^{y_0}}{|S(e^{y_0},\chi)|}\Big)^{2k^2}. 
\end{equation} 
\end{lemma}
\begin{proof}   Take $x=e^{y_0}$ and $f=\chi$ in section \ref{multfns}. By \eqref{Hal1} and the hypothesis we have 
$|\phi| \ll e^{y_0}/|S(e^{y_0},\chi)|\leq y_0^{\frac{1}{100}}$ as desired, and moreover that 
$M\le \frac 1{100} \log  y_0 + \log\log y_0+ O(1)$.  


To prove \eqref{lem231} we may clearly suppose that $|y-y_0| \le y_0/2$, in which case 
\eqref{lem231} follows immediately from \eqref{Hal3}.

Next, by Cauchy-Schwarz (and using $|1-\chi_{\phi}(p)|^2 \le 2 (1-\text{Re}(\chi_\phi(p)))$) 
 \begin{align*}
\sum_{p\le e^{y_0}} \frac{|1-\chi_\phi(p)|}{p}& \le \Big(\sum_{p\le e^{y_0}} \frac 1p \Big)^{\frac 12} \Big( \sum_{p\le e^{y_0}}  \frac{|1-\chi_\phi(p)|^2}{p} \Big)^{\frac 12} \le (\log y_0 +O(1))^{\frac 12} (2M)^{\frac 12}\\ &\le \frac{\log y_0}{7} +O(1).
 \end{align*}
Using this in \eqref{Hal2}, we obtain \eqref{lem232}.  

Finally, suppose that $\chi$ has order $k$.  The triangle inequality \eqref{TrIneq} gives 
$$ 
M=\sum_{p\le e^{y_0}} \frac{1-\text{Re }\chi_\phi(p)}{p} \ge \frac{1}{k^2} \sum_{p\le e^{y_0}} \frac{1- \text{Re }(\chi(p)p^{-i\phi})^k}{p} 
\ge \frac{1}{k^2} \sum_{p\le e^{y_0}} \frac{1-\cos (k\phi \log p)}{p}. 
$$ 
Using the prime number theorem it follows that $\phi \ll \exp(k^2 M)/y_0$ which yields the final assertion \eqref{lem233} of the lemma.  
\end{proof} 

Combining Lemmas \ref{Hadlem}, \ref{lem2.2}, and \ref{lem2.3}, we arrive at the following 
proposition.  

\begin{proposition} 
\label{prop2.4}  Let $y_0$ be large with $|S(e^{y_0},\chi)| =: e^{y_0}/N \ge e^{y_0}y_0^{-\frac{1}{100}}$, and let 
$\phi$ be as in Lemma \ref{lem2.3}.   If $c N^6/y_0 \le \lambda \le \frac 12$ for a suitably large constant $c$, then there exists $|\xi| \le 2\lambda \sqrt{(\log q)/y_0}$ such that 
\begin{equation} 
\label{2.45} 
\sum_{\rho} \frac{\lambda}{|1+\lambda+i\phi +i\xi -\rho|^2} \ge \frac{y_0}{4}. 
\end{equation} 
\end{proposition} 
\begin{proof} Put $T=\lambda/y_0$, and note that $T\le 1/(2y_0)<1$.  Using Lemma \ref{lem2.3} we find that 
\begin{align*}
\sqrt{2\pi T} &\int_{-\infty}^{\infty} \frac{S(e^y,\chi_\phi)}{e^y} \exp\Big(\lambda y-\frac{T}{2}y^2\Big) dy 
\\
&= \sqrt{2\pi T} \exp\Big(\frac{\lambda y_0}{2}\Big) \int_{-\infty}^{\infty}  \Big(\frac{S(e^{y_0},\chi_{\phi})}{e^{y_0} } + 
O\Big(\frac{1+|y-y_0|^{\frac 13}}{y_0^{\frac 13}}\Big) \Big) \exp\Big(-\frac T2(y-y_0)^2\Big) dy.
\end{align*} 
A little calculation, together with \eqref{lem232}, gives that this equals 
\begin{equation*} 
2\pi \exp\Big(\frac{\lambda y_0}{2} \Big) \Big( \frac{(1+i\phi)S(e^{y_0},\chi)}{e^{y_0(1+i\phi)}} + O\Big(\frac{1}{(\lambda y_0)^{\frac 16}}\Big)\Big), 
\end{equation*} 
which, by our assumed lower bound on $\lambda$, is in magnitude $\ge \pi \exp(\frac{\lambda y_0}2)/N$.

Thus, by Lemma \ref{lem2.2}, we see that 
\begin{align*}
\pi \exp\Big( \frac{\lambda y_0}{2} \Big)/N &\le \int_{-\infty}^{\infty} \frac{|L(1-\lambda+i\phi + i\xi,\chi)|}{|1-\lambda+i\xi|} \exp\Big(-\frac{\xi^2}{2T}\Big) d\xi \\
&\le \Big(\max_{\xi \in {\mathbb R}} \frac{|L(1-\lambda+i\phi +i\xi,\chi)|}{|1-\lambda+i\xi|} \exp\Big(-\frac{\xi^2}{4T}\Big) \Big) 
\int_{-\infty}^{\infty} \exp\Big(-\frac{\xi^2}{4T}\Big) d\xi, 
\end{align*} 
so that 
\begin{equation}
\label{2.8} 
\max_{\xi \in {\mathbb R}} \frac{|L(1-\lambda+i\phi +i\xi,\chi)|}{|1-\lambda+i\xi|} \exp\Big(-\frac{\xi^2}{4T}\Big) \ge 
\frac 12 \sqrt{\frac { \pi y_0} {\lambda}} \exp\Big( \frac{\lambda y_0}{2}\Big)/N. 
\end{equation} 

If Re$(s)=\sigma >0$ then note that 
$$ 
|L(s,\chi)| =\Big| s\int_1^{\infty} \frac{S(x,\chi)}{x^{s+1}}dx \Big| \le |s| \int_1^{\infty} \frac{\min(x,q)}{x^{\sigma+1}}dx = 
|s| \Big(\frac{q^{1-\sigma}-1}{1-\sigma} + \frac{q^{1-\sigma}}{\sigma} \Big). 
$$ 
Therefore, if $|\xi| > 2\lambda\sqrt{(\log q)/y_0}$ then 
$$ 
 \frac{|L(1-\lambda+i\phi +i\xi,\chi)|}{|1-\lambda+i\xi|} \exp\Big(-\frac{\xi^2}{4T}\Big) \le \Big|\frac{1-\lambda+i\phi+i\xi}{1-\lambda+i\xi}\Big| 
 \Big(\frac{q^{\lambda}-1}{\lambda}+\frac{q^{\lambda}}{1-\lambda}\Big)q^{-\lambda} \le \frac{2(1+2|\phi|)}{\lambda}.
 $$ 
 Since $|\phi| \ll N$, and $\lambda y_0 \ge cN^6$ for a suitably large constant $c$, we may check that the RHS above is 
 smaller than the RHS in \eqref{2.8}.  Therefore the maximum in the LHS of \eqref{2.8} is attained for some $\xi$ with $|\xi| \le 2\lambda\sqrt{(\log q)/y_0}$,  and at this point we have, by \eqref{2.8},
 \[
  |L(1-\lambda+i\phi +i\xi,\chi)| \gg  
 (\lambda y_0)^{1/3} \exp\Big( \frac{\lambda y_0}{2}\Big) .
\]
Using now the bound from Lemma \ref{Hadlem}, we conclude that 
$$ 
\sum_{\rho} \frac{\lambda}{|1+\lambda+i\phi +i\xi -\rho|^2} \ge \frac{y_0}{4}.
$$
\end{proof} 

\section{Proofs of Theorems \ref{main1} and \ref{main2} and the Corollaries}  

\begin{proof}[Proof of Theorem \ref{main1}]  We appeal to Proposition \ref{prop2.4}, taking there $y_0= \log x$, and 
let $\phi$, $\lambda$ and $\xi$ be as given there.   We then have the lower bound furnished by \eqref{2.45}.   Split the zeros $\rho$ into those with $|1+i\phi-\rho| \ge 40 \lambda (\log q)/\log x$ and those zeros lying closer to $1+i\phi$.  Note that if $|1+i\phi-\rho| \ge 40\lambda (\log q)/\log x$ then, using the triangle inequality,  
$$ 
|1+\lambda +i\phi +i\xi -\rho| \ge \Big|1+ 20 \lambda\frac{\log q}{\log x} +i\phi -\rho \Big| - 20\lambda \frac{\log q}{\log x} - |\xi| \ge \frac {9}{20}\Big|  1+ 20 \lambda\frac{\log q}{\log x} +i\phi -\rho \Big|. 
$$ 
Therefore the contribution of these zeros to the LHS of \eqref{2.45} is 
$$ 
\le 5 \lambda \sum_{\rho} \frac{1}{|1+20\lambda (\log q)/\log x +i\phi -\rho|^2} \le 
\frac{\log x}{4 \log q} \sum_{\rho} \text{Re }\Big(\frac{1}{1+20\lambda (\log q)/\log x +i\phi -\rho}\Big),
$$ 
as Re$(\rho)\leq 1$ for all such $\rho$.
But, arguing as in \eqref{2.41} and \eqref{2.42}, we see that the above is at most 
$$ 
\frac{\log x}{4 \log q} \Big( \frac{5}{9} \log q \Big) = \frac{5\log x}{36}. 
$$ 
We conclude that the contribution of the zeros with $|1+i\phi - \rho| \le 40 \lambda (\log q)/\log x$ to 
the LHS of \eqref{2.45} is at least $(\log x)/9$.   Since each such zero contributes at most $1/\lambda$, it 
follows that 
$$ 
\Big|\Big \{\rho: |1+i\phi-\rho| \le 40 \lambda\frac{\log q}{\log x} \Big\}\Big | \ge \lambda \frac{\log x}{9}.
$$
The theorem follows upon setting $L= 40\lambda \log x$.
\end{proof} 

\begin{proof}[Proof of Theorem \ref{main2}]  We follow the argument above, now making use of the bound \eqref{lem233} which gives $|\phi| \le (cN)^{2k^2}/\log x$.  Therefore if now $\lambda \ge (cN)^{2k^2}/\log x$ ($\geq |\phi|$)
and $|1-\rho| \ge 40 \lambda (\log q)/\log x$ then $|1+i\phi -\rho| \ge 39 \lambda (\log q)/\log x$ and the argument above shows that the contribution of these zeros to the LHS of \eqref{2.45} is bounded by $0.15 \log x$.  
Thus we conclude Theorem \ref{main2}.
\end{proof} 

Corollary \ref{cor2} follows upon taking $L=c\epsilon^{-6}$ in Theorem \ref{main1} and $L=(c/\epsilon)^{2k^2}$ 
in Theorem \ref{main2}.
Corollaries \ref{cor4} and \ref{cor3} follow  upon taking $L=(\epsilon^2\log q)/4$ in Theorems \ref{main2} and \ref{main1}, respectively.

\section{Locating zeros: Proof of Theorem \ref{thm1.6}} 

Choosing $L= (c/\eta)^{2k^2}$ for a suitably large constant $c$, we find by Theorem \ref{main2} that 
for each $\chi_j \pmod{q_j}$ there is a zero of $L(s,\chi_j)$ satisfying $s= 1+ w_j/\log q_j$ with $|w_j| \le C_1(\eta)$ 
for a suitable constant $C_1(\eta)$.  Since the region $|w|\le C_1(\eta)$ is compact, we can extract from the 
sequence $w_j$ a convergent subsequence.  Now take $z_1$ to be the limiting value of $w_j$ from this 
convergent subsequence.  

By restricting to the subsequence above, let us suppose that we now have a sequence of characters $\chi_j$ of order $k$ with $L(s,\chi_j)$ having a zero satisfying $1+(z_1+o(1))/\log q_j$.    Now from the argument of \eqref{2.41} and \eqref{2.42} we may see that for any $L$-function there are at most a bounded number of 
zeros of the form $1+w/\log q$ with $|w| \le 1+|z_1|$.  Therefore, by appealing to Theorem \ref{main2} 
with a suitably large value of $L$, we may conclude that $L(s,\chi_j)$ has a zero of the form $s=1+w_j/\log q_j$ with $|z_1|+1 \le |w_j|  \le C_2(\eta)$ for some suitably large $C_2(\eta)$.  Since this region is again compact, we can once again extract a subsequence of characters for which $w_j$ converges, and we call one such limiting  value $z_2$.  

Proceeding in this manner, we obtain Theorem \ref{thm1.6}.

\section{Relations among characters with large partial sums} 

\noindent We begin by showing that if a multiplicative function $f$ is at a small 
distance from the function $n^{i\phi}$ then the partial sums of $f$ get large in 
suitable ranges.  

\begin{proposition}\label{prop6.1} Let $f$ be a multiplicative
function with $|f(n)|\leq 1$ for all $n$.    Let $x$ be large, and put 
$\lambda=M+\log(1+|\phi|)+c$ where $M=M_f(x)$ and $\phi=\phi_f(x)$ are as in Section 2, 
and $c$ is a large constant. Then  there exists $y$ in the
range $x^{1/(\lambda e^\lambda)}\leq y\leq x$ such that
$$
\Big|\sum_{n\leq y} f(n)\Big| \gg \frac{e^{-M} }{|1+i\phi|}\ y.
$$
\end{proposition}

\begin{proof}  By \eqref{Ftodist} we know that 
$$ 
\Big|F\Big(1+ \frac{\lambda}{\log x} + i\phi \Big)\Big| \asymp \frac{\log x}{\lambda} \exp(-{\mathbb D}(f,n^{i\phi};x^{1/\lambda})^2) 
\gg \frac{\log x}{\lambda} e^{-M}. 
$$ 
On the other hand, with $\delta=\lambda/\log x$, 
$$
\sum_{n=1}^{\infty} \frac{f(n)}{n^{1+\delta+i\phi}} = (1+\delta+i\phi)
\int_1^\infty \frac{1}{y^{2+\delta+i\phi}} \sum_{n\leq y} f(n) dy.
$$
Put $\eta = 1/(\lambda e^{\lambda})$.  Assuming that $|\sum_{n\leq y} f(n)|\leq e^{-\lambda} y$ for all
$x^{\eta} \leq y\leq x$, and using the  trivial bound $|\sum_{n\leq
y} f(n)|\leq y$ otherwise, we find that the right hand side above is bounded in size by 
\begin{align*} 
 &\leq |1+\delta + i\phi| \Big( \int_1^{x^{\eta}} \frac{dy}{y^{1+\delta}} + e^{-\lambda}
\int_{x^\eta }^x \frac{dy}{y^{1+\delta}} +\int_x^\infty
\frac{dy}{y^{1+\delta}} \Big)  \\
&= |1+\delta+i\phi| \frac{\log x}\lambda \left( (1 -e^{-\eta \lambda})  + e^{-\lambda}(e^{-\eta
\lambda} - e^{-\lambda}) + e^{-\lambda}  \right)\\
& \leq 6 (1+|\phi|)e^{-\lambda} \ \frac{\log
x}\lambda .
\end{align*}
This yields a contradiction, provided $c$ is sufficiently large.


\end{proof}

Our next result shows that if the partial sums of two completely multiplicative functions get large, then the 
product of these functions also has large partial sums.  Corollary \ref{MultChars} follows immediately from 
this result.   

\begin{theorem} \label{multis}  Let $f_1$ and $f_2$ be completely multiplicative functions with $|f_1(n)|$ and $|f_2(n)|$ bounded 
by $1$ for all $n$.  Suppose that $\eta$ is a positive real number and $x_1$, and $x_2$ are such such that (for $j=1$, $2$) 
\begin{equation} \label{fjlower}
\Big|   \sum_{n\leq x_j} f_j(n)\Big| \geq \eta x_j   . 
\end{equation}
Then, with $\xi = c \eta^6$ for a suitable absolute constant $c>0$, there exists $x\geq  (\min (x_1,x_2) )^{\xi}$ such that, for some absolute constant $c>0$, 
$$
\Big|  \sum_{n\leq x} f_1(n) f_2(n) \Big| \geq \xi x.  
$$
\end{theorem}

\begin{proof} 
 By \eqref{Hal1} there exists $\phi_1, \phi_2$ with $|\phi_j|\ll
1/\eta$   such that
$$
\mathbb D(f_j,n^{i\phi_j};x_j)^2    \leq \log (1/\eta) + \log\log (1/\eta) +O(1).
$$
Let $X=\min\{ x_1,x_2\},\  f =f_1f_2$ and $\phi=\phi_1+\phi_2$. Since 
$\mathbb D(f_j,n^{i\phi_j};X)\leq \mathbb D(f_j,n^{i\phi_j};x_j)$,  the triangle
inequality \eqref{TrIneq} gives
$$
\mathbb D(f ,n^{i\phi};X) \leq    \mathbb D(f_1,n^{i\phi_1};X) + \mathbb D(f_2,n^{i\phi_2};X),
$$
so that $\mathbb D(f,n^{i\phi};X)^2\leq 4\log (1/\eta) + 4\log\log
(1/\eta) +O(1)$. The result now follows from Proposition \ref{prop6.1}.
\end{proof}

Another variant of the argument of Theorem \ref{multis} is the following.  Suppose $f$ is completely multiplicative 
with $|f(n)|\le 1$ and $|\sum_{n\le x} f(n) |\ge \eta x$.  Then for any natural number $k$, there exists 
$y\ge x^{c\eta^{2k^2}}$ (for a suitable absolute constant $c>0$) with 
$$ 
\Big| \sum_{n\le y} f(n)^k \Big| \ge c \eta^{2k^2} y. 
$$ 
To see this, note that our hypothesis on $f$ implies (as in Theorem \ref{multis})  that 
$$
{\mathbb D}(f,n^{i\phi};x)^2 \le \log (1/\eta) + \log \log (1/\eta) + O(1)
$$ 
for some $|\phi| \ll 1/\eta$.  By the triangle inequality it follows that  
${\mathbb D}(f^k,n^{ik\phi};x) \le k {\mathbb D}(f,n^{i\phi};x)$.  Now we invoke Proposition \ref{prop6.1}, and 
obtain the stated conclusion.  One application of this variant is that (stated informally) if a small power of a 
character $\chi$ equals a non-principal character of small conductor, then one can obtain cancelations in the 
character sums for $\chi$.



\section{Producing many quadratic residues below $p^{\frac 14}$} 

\noindent Corollary \ref{quadres} follows immediately from the Burgess bound together with the following general result on 
completely multiplicative functions taking values in $[-1,1]$, which largely follows from the 
work in \cite{GSSpectrum}.  

 \begin{proposition}  Let $x$ be large, and let $f$ be a completely multiplicative function with
$-1\leq f(n)\leq 1$ for all $n$.  Suppose that 
$$ 
\sum_{n\le x} f(n) = o(x). 
$$ 
Then for $1/\sqrt{e}\le \alpha \le 1$ we have 
$$
\frac{1}{x^\alpha} \Big| \sum_{n\le x^\alpha} f(n) \Big|  \leq \max\Big( |\delta_1|, \frac 12+2(\log \alpha)^2\Big) +o(1), 
$$
where 
$$
\delta_1= 1-2 \log (1+\sqrt{e}) + 4\int_1^{\sqrt{e}} \frac{\log t}{t+1} dt = - 0.656999\ldots. 
$$ 
\end{proposition} 

 \begin{proof} We make free use of the work in \cite{GSSpectrum}.  Put $y=\exp((\log x)^{\frac 23})$ and 
 let $g$ be the completely multiplicative function defined by $g(p)=1$ for $p\le y$ and $g(p)=f(p)$ for $p>y$.  
 Then (see page 439 of \cite{GSSpectrum}) for $x^{1/\sqrt{e}} \le z \le x$ we have 
 $$ 
 \sum_{n\le z} f(n) = \Big( \prod_{p\le y} \Big(1-\frac 1p\Big) \Big(1-\frac{f(p)}{p}\Big)^{-1} \Big) \sum_{n\le z} g(n) +o(z).
 $$ 
 If now 
 $$ 
 \prod_{p\le y} \Big(1-\frac 1p\Big) \Big(1-\frac{f(p)}{p}\Big)^{-1} \le \frac 1{10},
 $$ 
 then the result follows at once.  So let us assume that the product above is at least $\frac 1{10}$, so that 
 we have 
 $$ 
 \sum_{n\le x} g(n) = o(x).
 $$ 
 
 Now we may pass from mean values of multiplicative functions to solutions of integral equations as in 
 \cite{GSSpectrum}, and use the results established there.  Put 
 $$ 
 \tau(\alpha) = \sum_{p\le x^{\alpha}} \frac{1-g(p)}{p}. 
 $$ 
 By inclusion-exclusion (see Proposition 3.6 of \cite{GSSpectrum}) we have 
 $$ 
o(x)= \sum_{n\le x} g(n) \ge x(1 -\tau(1) +o(1)), 
$$ 
so that $\tau(1) \ge 1+o(1)$ and more generally 
$$ 
\tau(\alpha) \ge 1 -2 \sum_{x^{\alpha} \le p\le x} \frac{1}{p} = 1 +2\log \alpha + o(1).
$$ 
Applying Theorem 5.1 of \cite{GSSpectrum}, if $\tau(\alpha)  \geq 1$ then
$$
\Big|\sum_{n\le x^{\alpha}} g(n) \Big| \leq (|\delta_1| +o(1)) x^{\alpha}.
$$
If $1+2\log \alpha \le \tau(\alpha)\le 1$, then an inclusion-exclusion argument (see again Proposition 3.6 of \cite{GSSpectrum}) 
gives 
$$ 
\Big| \sum_{n\le x^{\alpha}} g(n) \Big| \le \Big(1-\tau(\alpha) + \frac{\tau(\alpha)^2}{2} +o(1)\Big) x^{\alpha} 
\le \Big( \frac 12+ 2 (\log \alpha)^2 +o(1) \Big) x^{\alpha}.
$$ 
The proposition follows. 
%
\end{proof}


There is some scope to improve the bound in Proposition 7.1, especially when $\alpha$ is close to $1$.  
Here Lipschitz estimates like \eqref{Hal3} show that 
$$ 
x^{-\alpha} \sum_{n\le x^{\alpha}} f(n)  \ll (1-\alpha)^{1-\frac{2}{\pi} +o(1)},
$$ 
which is plainly better than the bound in Proposition 7.1 for $\alpha$ sufficiently close to $1$.  




\bibliographystyle{plain}

\end{document}